\documentclass[12pt]{article}
\usepackage{amsmath}
\usepackage{amssymb}
\usepackage{amsfonts}
\usepackage{amsmath,amsthm,amssymb}
\allowdisplaybreaks

\newcommand{\ds}{\displaystyle}

\theoremstyle{plain}
\newtheorem{theorem}{Theorem}[section]
\newtheorem{remark}[theorem]{Remark}
\newtheorem{lemma}[theorem]{Lemma}
\newtheorem{proposition}[theorem]{Proposition}

\def \[{\begin{equation}}
\def \]{\end{equation}}

\newif \ifLastSection \LastSectionfalse

\numberwithin{equation}{section}

\newcommand{\R}{{\mathbb R}}

\begin{document}
\title{   The sharp existence of constrained minimizers for the $L^2$-critical Schr\"{o}dinger-Poisson system and Schr\"{o}dinger equations}

\author{Hongyu Ye\thanks{a:Partially  supported  by  NSFC No: 11501428,  NSFC No: 11371195} \thanks{b:
E-mail address: yyeehongyu@163.com} \\ \small {College of Science, Wuhan University of Science and Technology,}
\\
\small{ Wuhan 430065, P. R. China}}
\date{}
\maketitle

\begin{abstract}
In this paper, we study the existence of minimizers for a class of constrained minimization problems derived from the Schr\"{o}dinger-Poisson equations:
$$-\Delta u+V(x)u+(|x|^{-1}*u^2)u-|u|^\frac{4}{3}u=\lambda u,~~x\in\R^3$$
on the $L^2$-spheres $\widetilde{S}(c)=\{u\in H^1(\R^3)|~\int_{\R^3}V(x)u^2dx<+\infty,~|u|_2^2=c>0\}$. If $V(x)\equiv0$, then by a different method from Jeanjean and Luo [Z. Angrew. Math. Phys. 64 (2013), 937-954], we show that there is no minimizer for all $c>0$; If $0\leq V(x)\in  L^{\infty}_{loc}(\R^3)$ and $\lim\limits_{|x|\rightarrow+\infty}V(x)=+\infty$, then a minimizer exists if and only if $0<c\leq c^*=|Q|_2^2$, where $Q$ is the unique positive radial solution of $-\Delta u+u=|u|^{\frac{4}{3}}u,$ $x\in\R^3$. Our results are sharp. We also extend some results to constrained minimization problems on $\widetilde{S}(c)$ derived from Schr\"{o}dinger operators:
$$F_\mu(u)=\frac{1}{2}\ds\int_{\R^N}|\nabla u|^2-\frac{\mu}2\ds\int_{\R^N}V(x)u^2-\frac{N}{2N+4}|u|^\frac{2N+4}{N}$$
where $0\leq V(x)\in L^{\infty}_{loc}(\R^N)$ and $\lim\limits_{|x|\rightarrow+\infty}V(x)=0$. We show that if $\mu>\mu_1$ for some $\mu_1>0$, then a minimizer exists for each $c\in(0,c^*)$.

\end{abstract}

\section{ Introduction and main result}
In the past years, the following Schr\"{o}dinger-Poisson type equation with pure power nonlinearities:
\begin{equation}\label{1.1}
i\psi_t+\Delta\psi-(|x|^{-1}*|\psi|^2)\psi+|\psi|^{p-2}\psi=0~~\hbox{in}~\R^3
\end{equation}
has been well studied, where $\psi(x,t):\R^3\times[0,T)\rightarrow\mathbb{C}$ is the wave function, $2<p<6$ and $|x|^{-1}*u^2$ denotes a repulsive nonlocal Coulombic potential. Equation \eqref{1.1} arises from approximation of the Hartree-Fock equation which describes a quantum mechanical of many particles, see e.g. \cite{bgg,ls,li,m}.

One usually searches for the existence of stationary solutions $\psi(x,t)=e^{-i\lambda t}u(x)$ to \eqref{1.1}, where $\lambda\in\R$ and $u:\R^3\rightarrow\R$ is a function to be founded, then \eqref{1.1} is reduced to be the following system
\begin{equation}\label{1.2}\left\{%
\begin{array}{ll}
    -\Delta u+\phi_uu-\lambda u=|u|^{p-2}u, & \hbox{$x\in\R^3$}, \\
   -\Delta \phi_u=4\pi u^2, & \hbox{$x\in \R^3$},\\
\end{array}%
\right.\end{equation}
where $\phi_u=|x|^{-1}*u^2$ is understood as the scalar potential of the electrostatic field uniquely generated by the charge density $u^2$. The case where $\lambda\in\R$ is a fixed and assigned parameter has been extensively studied, see \cite{aru,apa,dm,k1,k2,r1,r2} and the reference therein. In such case, solutions of \eqref{1.2} correspond to critical points of the following functional
$$\Psi(u)=\frac{1}{2}\ds\int_{\R^3}|\nabla u|^2-\frac{1}{2}\lambda\ds\int_{\R^3}u^2+\frac{1}{4}\ds\int_{\R^3}\phi_uu^2
-\frac{1}{p}\ds\int_{\R^3}|u|^{p}.$$
However, nothing can be given a priori on the $L^2$-norm of the solutions. Since the physicists are usually interested in ``$\hbox{normalized~solutions}"$, i.e. solutions with prescribed $L^2$-norm, it is interesting for us to study whether \eqref{1.2} has normalized solutions. For any fixed $c>0$, a solution of \eqref{1.2} with $|u|_2^2=c$ can be viewed as a critical point of the following functional
\begin{equation}\label{1.3}
E(u)=\frac{1}{2}\ds\int_{\R^3}|\nabla u|^2+\frac{1}{4}\ds\int_{\R^3}\phi_uu^2-\frac{1}{p}\ds\int_{\R^3}|u|^{p}
\end{equation}
constrained on the $L^2$-spheres in $H^1(\R^3)$:
$$S(c)=\{u\in H^1(\R^3)|~|u|_2^2=c,~c>0\}.$$
Note that the frequency $\lambda$ now is no longer by imposed but instead appears as a Lagrange multiplier. We call $(u_c,\lambda_c)\in H^1(\R^3)\times\R$ a couple of solution to \eqref{1.2} if $u_c$ is a critical point of $E|_{S(c)}$ and $\lambda_c$ is the associated Lagrange parameter.

Set \begin{equation}\label{1.4}
e_c:=\inf\limits_{u\in S(c)}E(u).
\end{equation}
Then minimizers of $e_c$ are exactly constrained critical points of $E(u)$ on $S(c)$. It seems that $p=\frac{10}{3}$ is the $L^2$-critical exponent for \eqref{1.4}, i.e. for all $c>0$, $e_c>-\infty$ if $2<p<\frac{10}{3}$ and $e_c=-\infty$ if $\frac{10}{3}<p<6$. There have been some papers considering the existence of minimizers of $e_c$, see e.g. \cite{bs2,bs1,cdss,jl,ss}. In \cite{bs2}, Bellazzini and Siciliano proved that there exists at least one minimizer for \eqref{1.4} provided that $p\in(2,3)$ and $c>0$ is small enough. In \cite{ss}, Sanchez and Soler considered \eqref{1.4} with $p=\frac{8}{3}$ and showed that there exists $c_0>0$ small such that minimizers exist for any $c\in(0,c_0)$. In \cite{bs1}, Bellazzini and Siciliano proved that \eqref{1.4} admits at least one minimizer if $p\in(3,\frac{10}{3})$ and $c>0$ is large enough. In \cite{jl}, Jeanjean and Luo showed the sharp nonexistence results for \eqref{1.4} with $p\in[3,\frac{10}{3}]$, i.e. for $p\in(3,\frac{10}{3})$, there exists $c_1>0$ such that \eqref{1.4} has a minimizer if and only if $c\geq c_1$. When $p=3$ or $p=\frac{10}{3}$, no minimizer exists for all $c>0$. Furthermore, when $p=3$, Kikuchi in \cite{k1} considered a minimization problem $e_{\Lambda,c}:=\inf\limits_{u\in S(c)}E_\Lambda(u),$ where
$$E_\Lambda(u)=\frac{1}{2}\ds\int_{\R^3}|\nabla u|^2+\frac{1}{4}\ds\int_{\R^3}\phi_uu^2-\frac{\Lambda}{3}\ds\int_{\R^3}|u|^{3}$$
 and proved that there exists $\Lambda_0>0$ such that $e_{\Lambda,c}$ is attained at any $c>0$ and $\Lambda>\Lambda_0$. For $\frac{10}{3}<p<6$, problem \eqref{1.4} does not work. It has been proved in \cite{bjl} that there exists at least one critical point of $E(u)$ restricted to $S(c)$ with a minimax characterization.

When $p=\frac{10}{3}$, as far as we know, in the literature there is just one paper \cite{jl} concerning about such case, which showed that for some $c_2\in(0,+\infty)$, $e_c=0$ if $c\in(0,c_2)$ and $e_c=-\infty$ if $c>c_2$. However, they did not give the accurate expression of $c_2$ and $e_{c_2}$ is unknown yet. In this paper, we use an alternative method to study problem \eqref{1.4} with $p=\frac{10}{3}$ and succeeded in obtaining a threshold value $c^*$ satisfying that $e_c=0$ if $c\in(0,c^*]$ and $e_c=-\infty$ if $c>c^*$. Our result is sharp.

Recall from \cite{gnn}\cite{k} that the following scalar field equation
\begin{equation}\label{1.11}
-\Delta u+u=|u|^{\frac{4}{N}}u,~~~x\in\R^N,
\end{equation}
up to translations, has a unique positive least energy and radially symmetric solution $Q\in H^1(\R^N)$, where $N\geq1$. It is proved in \cite{gnn} that $Q$ is monotonically decreasing away from the origin and
\begin{equation}\label{1.5}
Q(x),~|\nabla Q(x)|=O(|x|^{-\frac{1}{2}}e^{-|x|})~~~~\hbox{as}~~|x|\rightarrow+\infty.
\end{equation}
Moreover,
\begin{equation}\label{1.13}
\ds\int_{\R^N}|Q|^{\frac{2N+4}{N}}=\frac{N+2}{N}\ds\int_{\R^N}|\nabla Q|^{2}=\frac{N+2}{2}\ds\int_{\R^N}|Q|^{2}.
\end{equation}
We recall from \cite{w} the following Gagliardo-Nirenberg inequality with the best constant:
\begin{equation}\label{1.6}
\ds\int_{\R^N}|u|^{\frac{2N+4}{N}}\leq \frac{N+2}{N|Q|_2^{\frac{4}{N}}}\ds\int_{\R^N}|\nabla u|^2\left(\ds\int_{\R^N}|u|^2\right)^{\frac{2}{N}},
\end{equation}
where equality holds for $u=Q$. Then we get our main result:

\begin{theorem}\label{th1.1}\ \ Let $p=\frac{10}{3}$ and $c^*:=|Q|_2^2$, then $e_c=\left\{%
\begin{array}{ll}
    0, & \hbox{$0<c\leq c^*$}, \\
   -\infty, & \hbox{$c>c^*$}.\\
\end{array}%
\right.$
Moreover,

(1)~~$e_c$ has no minimizer for all $c>0$.

(2)~~for any $c\in(0,c^*]$, there is no critical point of $E(u)$ constrained on $S(c)$.
\end{theorem}

In order to obtain minimizers, we try to add a nonnegative perturbation term to the right-hand side of \eqref{1.3}, i.e. we consider the following new functional
\begin{equation}\label{1.7}
I(u)=\frac{1}{2}\ds\int_{\R^3}|\nabla u|^2+\frac12\ds\int_{\R^3}V(x)u^2+\frac{1}{4}\ds\int_{\R^3}\phi_uu^2
-\frac{3}{10}\ds\int_{\R^3}|u|^{\frac{10}{3}},
\end{equation} where the potential $V:\R^3\rightarrow\R$ is assumed to satisfy the following condition:
$$V(x)\in L^{\infty}_{loc}(\R^3),~~\inf\limits_{x\in\R^3}V(x)=0~~\hbox{and}~
~\lim\limits_{|x|\rightarrow+\infty}V(x)=+\infty.\eqno{(V_1)}$$
Based on $(V_1)$, we introduce a Sobolev space $\mathcal{H}=\left\{u\in H^1(\R^3)|~\int_{\R^3}V(x)u^2<+\infty\right\}$ with its associated norm $$\|u\|_{\mathcal{H}}=\left(\int_{\R^3}[|\nabla u|^2+(1+V(x))u^2]\right)^{\frac{1}{2}}.$$
Set
\begin{equation}\label{1.8}
I_c:=\inf\limits_{u\in\widetilde{S}(c)}I(u),
\end{equation}
where $\widetilde{S}(c)=\left\{u\in \mathcal{H}|~|u|_2^2=c,c>0\right\}.$

Our main results are as follows:

\begin{theorem}\label{th1.2}\ \ Suppose that $V(x)$ satisfies $(V_1)$.

(1)~~If $0<c\leq c^*$, there exists at least one minimizer for \eqref{1.8} and $I_c>0$.

(2)~~If $c>c^*,$ $I_c=-\infty$ and there is no minimizer for \eqref{1.8}.
\end{theorem}

We give the main idea of the proof of Theorem \ref{th1.2}. We notice that under $(V_1)$, the embedding $\mathcal{H}\hookrightarrow L^p(\R^3)$, $p\in[2,6)$ is compact (see e.g. \cite{bw}). To prove Theorem \ref{th1.2}, it is enough to show that each minimizing sequence of $I_c$ is bounded in $\mathcal{H}$. By \eqref{1.6}, when $c<c^*$, the boundedness of minimizing sequences for $I_c$ can be easily obtained, then minimizers exist. However, when $c=c^*$, the boundedness of the minimizing sequence cannot be similarly proved by using \eqref{1.6}. To do so, we need to consider the behavior of the function $c\rightarrow I_c$. The properties can be summarized in the following theorem.

\begin{proposition}\label{prop1.3}\ \
Suppose that $V(x)$ satisfies $(V_1)$. Then

(1)~~$\frac{I_c}{c^2}$ is strictly decreasing on $(0,c^*]$.

(2)~~$c\mapsto I_c$ is continuous on $(0,c^*]$.

(3)~~$\lim\limits_{c\rightarrow0^+}I_{c}=0.$
\end{proposition}

Proposition \ref{prop1.3} (1) gives us a cue to choose a sequence of minimizers for $I_{c}$ with $c<c^*$ to be the desired bounded minimizing sequence for $I_{c^*}$. We succeeded in doing so by taking $c_n=c^*(1-\frac1{2n})\nearrow c^*$ and the related minimizer sequence $\{u_n\}\in \widetilde{S}(c_n)$ as the special bounded minimizing sequence for $I_{c^*}$. Then Theorem \ref{1.2} (1) is proved. For the case $c>c^*$, we take a suitable test function in $\widetilde{S}(c)$ to show $I_c=-\infty$, which is also used in \cite{gs} to consider the Gross-Pitaevskii energy functional.

Another aim of this paper is to extend some above results to the following constrained minimization problem related to Schr\"{o}dinger operators:
$$
f_c:=\inf\limits_{u\in S(c)}F(u),
$$
where
\begin{equation}\label{1.10}
F(u)=\frac{1}{2}\ds\int_{\R^N}|\nabla u|^2-\frac{N}{2N+4}\ds\int_{\R^N}|u|^{\frac{2N+4}{N}}\end{equation}
and $N\geq1$.
\begin{theorem}\label{th1.4}~~For $Q$ and $c^*$ given in \eqref{1.11} and Theorem \ref{th1.1} respectively.

$(1)$~~$f_c=0$ if $c\in(0,c^*]$ and $f_c=-\infty$ if $c>c^*$.

$(2)$~~For all $c>0$ and $c\neq c^*$, $f_c$ has no minimizer.

(3)~~For $0<c<c^*$, $F(u)$ has no critical point constrained on $S(c)$.

(4)~~Set $$A=\{u\in S(c^*)|~u~\hbox{is~a~critical~point~of}~F\}.$$
Then $Q\in A$ and $A=\{(\sqrt{-\lambda_{c^*}})^{\frac{N}{2}}Q(\sqrt{-\lambda_{c^*}}x)|~\lambda_{c^*}<0\}$.

(5)~~For $c>c^*$, if $u_c$ is a critical point of $F$ constrained on $S(c)$, then $u_c$ is signchanging.
\end{theorem}

\begin{remark}\label{rem1.6}
If we replace $\frac{2N+4}{N}$ by a general $p$ in \eqref{1.10}, the case $p<\frac{2N+4}{N}$ and $p>\frac{2N+4}{N}$ have been well studied see e.g. \cite{stuart} and \cite{jean} respectively. Theorem \ref{th1.4} fills the gap.
\end{remark}

In order to obtain minimizers, the case with a potential term similar to \eqref{1.7} was considered in \cite{gs}. In this paper, we try to add a nonpositive perturbation term to the right-hand side of \eqref{1.10}, i.e. we consider the following functional:
\begin{equation}\label{1.12}
F_\mu(u)=\frac{1}{2}\ds\int_{\R^N}|\nabla u|^2-\frac{\mu}{2}\ds\int_{\R^N}V(x)u^2-\frac{N}{2N+4}\ds\int_{\R^N}|u|^{\frac{2N+4}{N}},
\end{equation}
where $\mu>0$ and $V(x)$ satisfies
$$V(x)\in L^{\infty}_{loc}(\R^N),~~~~V(x)\geq0~~~\hbox{and}~~~\lim\limits_{|x|\rightarrow+\infty}V(x)=0.\eqno{(V_2)}$$
We consider the minimization problem
\begin{equation}\label{1.14}
f_\mu(c)=\inf\limits_{u\in S(c)}F_\mu(u).
\end{equation}
Recall in \cite{sw} that the following minimum problem:
\begin{equation}\label{1.15}
\mu_1:=\inf\left\{\ds\int_{\Omega}|\nabla u|^2|~u\in H_0^1(\Omega),\ds\int_{\Omega}V(x)u^2=1\right\}
\end{equation}
is achieved by some $\phi\in H_0^1(\Omega)$ with $\int_{\Omega}V(x)\phi^2=1$ and $\phi>0$ a.e. in $\Omega$, where $\Omega\subset\R^N$ is a bounded domain with smooth boundary and $V(x)\not\equiv0$ a.e. in $\Omega$.
\begin{theorem}\label{th1.5}
Suppose that $V(x)$ satisfies $(V_2)$ and $\mu>0$, then

$(1)$~~$f_\mu(c)=-\infty$ for $c>c^*$;

$(2)$~~if $\mu\geq\mu_1$, $f_\mu(c)\in(-\infty,0)$ for all $c\in(0,c^*]$, moreover, for each $0<c<c^*$, there exists a couple $(u_c,\lambda_c)\in S(c)\times\R^-$ satisfying the following equation:
$$-\Delta u-\mu V(x)u-|u|^{\frac4N}u=\lambda_c u,~~x\in\R^N$$
with $F_\mu(u_c)=f_\mu(c)$.
\end{theorem}

The proof of Theorem \ref{th1.4} is similar to that of Theorem \ref{th1.1}, in which \eqref{1.13} and the inequality \eqref{1.6} play important roles.\\

To prove Theorem \ref{th1.5}, since $F_\mu(u)$ is bounded from below and coercive on $S(c)$ for each $c\in(0,c^*)$, the main difficulty is to deal with a possible lack of compactness for minimizing sequences of $f_{\mu}(c).$ We try to use the concentration-compactness principle to check the compactness. To do so, a necessary step is to show that $f_{\mu}(c)<0,$ which can be proved by using the minimizer for \eqref{1.15} and restricting the range of $\mu$. Then we succeeded in excluding the vanishing case by using the decay property of $V(x)$ at infinity. To avoid the dichotomy case, we need to obtain a strong version of subadditivity inequality
\begin{equation}\label{1.16}
 f_{\mu}(c)<f_{\mu}(\alpha)+f_{\mu}(c-\alpha),~~~~\forall~0<\alpha<c<c^*.
  \end{equation}
 The scaling argument used in \cite{bs2,bs1,jl} to get \eqref{1.16} cannot be applied here since $F_{\mu}(u)$ is no more an autonomous functional. To overcome this difficulty, we note that for $u\in S_c$ and $\theta>1$ the only scaling: $u_\theta:=\theta u$ can be used in our case. By using $f_\mu(c)<0$ and such a scaling, we finally prove that \eqref{1.16} holds, then Theorem \ref{th1.5} is proved.

Throughout this paper, we use standard notations. For simplicity, we
write $\int_{\Omega} h$ to mean the Lebesgue integral of $h(x)$ over
a domain $\Omega\subset\R^N$. $L^{p}:= L^{p}(\R^{N})~(1\leq
p<+\infty)$ is the usual Lebesgue space with the standard norm
$|\cdot|_{p}.$ We use `` $\rightarrow"$ and `` $\rightharpoonup"$ to denote the
strong and weak convergence in the related function space
respectively. $C$ will
denote a positive constant unless specified. We use `` $:="$ to denote definitions. $B_r(x):=\{y\in\R^N|~|y-x|<r\}$. We denote a subsequence
of a sequence $\{u_n\}$ as $\{u_n\}$ to simplify the notation unless
specified.

The paper is organized as follows. In $\S$ 2, we prove Theorem \ref{th1.1}. In $\S$ 3, we prove our main result Theorem \ref{th1.2} and Proposition \ref{prop1.3}. In $\S$ 4, we prove Theorem \ref{th1.4} and Theorem \ref{th1.5}.

\section{Proof for Theorem \ref{th1.1}}
For simplicity, we denote
\begin{equation}\label{2.11}
A(u):=\frac 12\ds\int_{\R^3}|\nabla u|^2,~~B(u):=\frac 14\ds\int_{\R^3}\phi_uu^2,~~C(u):=\frac{3}{10}\ds\int_{\R^3}|u|^{\frac{10}{3}}.
\end{equation}
Then for any $c>0$ and $u\in S(c)$, by \eqref{1.6}, we have
\begin{equation}\label{2.1}
C(u)\leq \left(\frac{c}{c^*}\right)^{\frac{2}{3}}A(u).
\end{equation}

\noindent $\textbf{Proof of Theorem \ref{th1.1}}$\,\,\
\begin{proof}~~The proof consists of four steps.\\

\noindent \textbf{Step~1.}~~$e_c=0$ for all $c\in(0,c^*]$.

By \eqref{2.1}, for each $c\in(0,c^*]$ and $u\in S(c)$, we see that
$$E(u)=A(u)+B(u)-C(u)\geq
\left(1-\left(\frac{c}{c^*}\right)^{\frac{2}{3}}\right)A(u)+B(u)\geq0.$$
Hence $e_c=\inf\limits_{S(c)}E(u)\geq0$.

On the other hand, set $u^t(x):=t^{\frac{3}{2}}u(tx)$ with $t>0$, then $u^t\in S(c)$ and $E(u^t)=t^2A(u)+tB(u)-t^2C(u)\rightarrow 0$ as $t\rightarrow 0^+$. Hence $e_c\leq 0$. So $e_c=0$ for all $c\in(0,c^*]$.\\

\noindent\textbf{Step~2}.~~For all $c>c^*$, $e_c=-\infty$ and then there is no minimizer for \eqref{1.4}.

For any $c>c^*$ and $t>0$, let $Q^t(x):=t^{\frac{3}{2}}\sqrt{\frac{c}{c^*}}Q(t x),$ where $Q$ is given in \eqref{1.11}. Then $Q^t\in S(c)$ and by \eqref{1.13}, we see that
$$
e_c\leq E(Q^t)
=\ds t^2\frac{c A(Q)}{c^*}\left[1-\left(\frac{c}{c^*}\right)^{\frac{2}{3}}\right]
+\ds t\left(\frac{c}{c^*}\right)^2B(Q)\rightarrow-\infty~~~\hbox{as}~~t\rightarrow+\infty
$$
since $c>c^*$. So $e_c=-\infty$ for all $c>c^*.$\\

\noindent\textbf{Step~3}.~~$e_c$ has no minimizer for $c\in(0,c^*]$.

By contradiction, we just suppose that there exists $c_0\in(0,c^*]$ such that $e_{c_0}$ has a minimizer $u_0\in S(c_0)$, i.e. $E(u_0)=e_{c_0}=0.$
Then by \eqref{2.1}, we have
$$A(u_0)+B(u_0)=C(u_0)\leq \left(\frac{c}{c^*}\right)^{\frac{2}{3}}A(u_0)\leq A(u_0),$$
which implies that $B(u_0)=0$. It is a contradiction. Therefore \eqref{1.4} admits no minimizer for any $c\in(0,c^*]$.\\

\noindent\textbf{Step~4}.~~For any $c\in(0,c^*]$, there is no critical point of $E(u)$ constrained on $S(c)$.

By contradiction, if for some $c\in(0,c^*]$, $E|_{S(c)}$ has a critical point $u_c$, i.e. $u_c\in S(c)$ and $(E|_{S(c)})'(u_c)=0$, then there is a
Lagrange multiplier $\lambda_c\in\R$ such that
\begin{equation}\label{2.12}
E'(u_c)-\lambda_cu_c=0~~~~\hbox{in}~~H^{-1}(\R^3).
\end{equation}
Then $u_c$ satisfies the following Pohozaev identity (see \cite{r1}):
\begin{equation}\label{2.13}
A(u_c)+5B(u_c)=3C(u_c)+\frac{3}{2}\lambda_cc.
\end{equation}
Hence by \eqref{2.1}-\eqref{2.13}, we see that
$$A(u_c)+\frac{1}{2}B(u_c)=C(u_c)\leq \left(\frac{c}{c^*}\right)^{\frac{2}{3}}A(u_c)\leq A(u_c),$$
which is impossible. Therefore the theorem is proved.
\end{proof}

\section{Proof of Theorem \ref{th1.2} and Proposition \ref{prop1.3}}
In this section, we consider the minimization problem \eqref{1.8}. We need the following compactness result, see e.g. \cite{bw}.

\begin{lemma}\label{lem3.1}~~
Suppose that $V(x)\in L^{\infty}_{loc}(\R^3)$ with $\lim\limits_{|x|\rightarrow +\infty}V(x)=+\infty$. Then the embedding $\mathcal{H}\hookrightarrow L^q(\R^3),~2\leq q<6$ is compact.
\end{lemma}
\begin{lemma}\label{lem3.2} (\cite{wi}, Vanishing lemma)~~Let $r>0$ and $2\leq q<2^*$. If $\{u_n\}$ is bounded in
$H^1(\R^N)$ and
$$\sup\limits_{y\in\R^N}\ds\int_{B_r(y)}|u_n|^q\rightarrow0,~~n\rightarrow+\infty,$$
then $u_n\rightarrow0$ in $L^s(\R^N)$ for $2<s<2^*$.
\end{lemma}

\noindent $\textbf{Proof of Theorem \ref{th1.2}}$\,\,\

\begin{proof}~~We complete our proof in four steps.\\

\noindent\textbf{Step~1}.~~For any $c\in(0,c^*)$, $I_c$ has a minimizer and $I_c>0$.

Using the same notations as in \eqref{2.11} and set
\begin{equation}\label{3.1}
D(u):=\frac{1}{2}\ds\int_{\R^3}V(x)u^2.
\end{equation}
By \eqref{2.1}, for any $u\in \widetilde{S}(c),$ since $V(x)\geq0$, we have
\begin{equation}\label{3.2}
I(u)\geq\left(1-\left(\frac{c}{c^*}\right)^{\frac{2}{3}}\right)A(u)
+B(u)+D(u)\geq0.
\end{equation}
Then $I_c\geq0$ for all $c\in(0,c^*]$.

For any $c\in (0,c^*)$, let $\{u_n\}\subset \widetilde{S}(c)$ be a minimizing sequence for $I_c$, then \eqref{3.2} implies that $\{u_n\}$ is bounded in $\mathcal{H}$, hence up to a subsequence, there may exist $u_c\in \mathcal{H}$ such that $u_n\rightharpoonup u_c$ in $\mathcal{H}$. By Lemma \ref{lem3.1}, $u_n\rightarrow u_c$ in $L^q(\R^3),$ $2\leq q<6,$ which implies that $|u_c|_2^2=c$, i.e. $u_c\in \widetilde{S}(c)$. So by the weak lower semicontinuity of the norm in $\mathcal{H}$, we have $I_c\leq I(u_c)\leq\liminf\limits_{n\rightarrow+\infty}I(u_n)=I_c,$ i.e. $u_c$ is a minimizer for $I_c$. So \eqref{1.8} admits at least one minimizer for $c\in(0,c^*)$ and it follows from \eqref{3.2} that $I_c>0$.\\

\noindent\textbf{Step~2}.~~The function $c\mapsto\frac{I_c}{c^2}$ is strictly decreasing on $(0,c^*)$.

For any $0<c_1<c_2<c^*$, by Step 1, there is $u_{1}\in\widetilde{S}(c_1)$ such that $I_{c_1}=I(u_1)>0$. Let $v:=\sqrt{\frac{c_2}{c_1}}u_1,$ then $v\in\widetilde{S}(c_2)$. Consider a new function $f:[1,+\infty)\rightarrow\R$ defined as follows:
$$
f(t)=(1-t)(A(u_1)+D(u_1))+(t-t^{\frac{2}{3}})C(u_1).
$$
Since
\begin{equation}\label{3.26}
f'(t)=-A(u_1)-D(u_1)+\left(1-\frac{2}{3t^{\frac{1}{3}}}\right)C(u_1)~~~~\hbox{and}~~~~
f''(t)=\frac{2}{9t^{\frac43}}C(u_1),
\end{equation}
$f'(t)$ is strictly increasing on $[1,+\infty)$. Then for any $t\geq1$, by \eqref{2.1} we have
$$f'(t)<\lim\limits_{t\rightarrow+\infty}f'(t)=-A(u_1)-D(u_1)+C(u_1)<0,$$
which implies that $f(t)<f(1)=0$ for all $t>1$. Hence
$$
I_{c_2}\leq I(v)=\ds\left(\frac{c_2}{c_1}\right)^2I(u_1)+\ds\frac{c_2}{c_1}f\left(\frac{c_2}{c_1}\right)
<\ds\left(\frac{c_2}{c_1}\right)^2I_{c_1},
$$
i.e. $\frac{I_{c_2}}{c_2^2}< \frac{I_{c_1}}{c_1^2}$, so $\frac{I_{c}}{c^2}$ is strictly decreasing on $(0,c^*)$.\\

\noindent\textbf{Step~3}.~~When $c=c^*$, a minimizer for \eqref{1.8} exists.

Let $c_n=c^*(1-\frac{1}{2n})$, then $c_n\nearrow c^*$. By Step 1, there exists a sequence $\{u_n\}\subset\widetilde{S}(c_n)$ such that $I(u_n)=I_{c_n}$. By Step 2, we see that
$$I_{c_n}<\frac{I_{c_1}}{c^2_1}c^2_n=4I_{\frac{c^*}2}\left(1-\frac{1}{2n}\right)^2\leq4I_{\frac{c^*}2},$$
i.e. $I_{c_n}$ is uniformly bounded. By \eqref{3.2}, we have $D(u_n)\leq I_{c_n}$, i.e. $\{\int_{\R^3}V(x)u_n^2\}$ is uniformly bounded.
It is enough to prove $\{A(u_n)\}$ is uniformly bounded. By contradiction, we just suppose that
\begin{equation}\label{3.7}
A(u_n)\rightarrow+\infty~~~\hbox{as}~~n\rightarrow+\infty.
\end{equation}
By \eqref{2.1} we have
\begin{equation}\label{3.8}
0\leq A(u_n)-C(u_n)\leq I_{c_n}\leq4I_{\frac{c^*}2},
\end{equation}
which implies that
$\lim\limits_{n\rightarrow+\infty}\frac{C(u_n)}{ A(u_n)}=1$, i.e. $\lim\limits_{n\rightarrow+\infty}C(u_n)=+\infty$. Let
\begin{equation}\label{3.9}
\varepsilon_n:=\frac{1}{C(u_n)^{\frac12}},
\end{equation}
then $\varepsilon_n\rightarrow0$ as $n\rightarrow+\infty$. Set
\begin{equation}\label{3.10}
w_n(x):=\varepsilon_n^{\frac32}u_n(\varepsilon_n x).
\end{equation}
Then $|w_n|_2^2=c_n$ and by \eqref{3.8}-\eqref{3.10},
\begin{equation}\label{3.11}
C(w_n)=1,~~~~~~1\leq A(w_n)\leq 1+4I_{\frac{c^*}2}\varepsilon_n^2.
\end{equation}
Denote
$\delta:=\lim\limits_{n\rightarrow+\infty}\sup\limits_{y\in\R^3}\int_{B_1(y)}|w_n|^2.$
If $\delta=0$, then by Lemma \ref{lem3.2}, $w_n\rightarrow 0$ in $L^q(\R^3)$, $\forall~q\in(2,6)$. Hence $C(w_n)\rightarrow0$, which contradicts \eqref{3.11}. Therefore, $\delta>0$. Then there exists a sequence $\{y_n\}\subset\R^3$ such that
 \begin{equation}\label{3.12}
 \ds\int_{B_1(y_n)}|w_n|^2\geq \frac{\delta}{2}>0.
 \end{equation}
 Set
$$
\widetilde{w}_n(x):=w_n(x+y_n)=\varepsilon_n^{\frac32}u_n(\varepsilon_n x+\varepsilon_n y_n).
$$
Then by \eqref{3.11}, $\{\widetilde{w}_n\}$ is a bounded sequence in $H^1(\R^3)$. We may assume that, up to a subsequence, there exists $w_0\in H^1(\R^3)$ such that
\begin{equation}\label{3.14}\widetilde{w}_n\rightharpoonup w_0~~\hbox{in}~H^1(\R^3),~~~~~~~\widetilde{w}_n\rightarrow w_0~~\hbox{in}~L^q_{loc}(\R^3),~q\in[1,6),~~~~w_n(x)\rightarrow w_0(x)~~\hbox{a.e.~in}~\R^3.
 \end{equation}
By \eqref{3.12}, we have $\int_{B_1(0)}|\widetilde{w}_n|^2\geq \frac{\delta}{2}>0$, then  $w_0\not\equiv0$ in $H^1(\R^3)$. By \eqref{3.2}, we have
\begin{equation}\label{3.17}
0\leq B(\widetilde{w}_n)=\varepsilon_n B(u_n)\leq \varepsilon_nI_{c_n}\rightarrow0~~\hbox{as}~n\rightarrow+\infty,
\end{equation}
i.e. $\lim\limits_{n\rightarrow+\infty}B(\widetilde{w}_n)=0$. However, by \eqref{3.14} and the Fatou's Lemma we see that $0<B(w_0)\leq \liminf\limits_{n\rightarrow+\infty}B(\widetilde{w}_n)=0$, which is a contradiction.
Therefore, $\{u_n\}$ is bounded in $\mathcal{H}$.

There exists $u_0\in \mathcal{H}$ such that $u_n\rightharpoonup u_0$ in $\mathcal{H}$. By Lemma \ref{lem3.1}, $u_n\rightarrow u_0$ in $L^2(\R^N)$, so $u_0\in \widetilde{S}(c^*)$ and
\begin{equation}\label{3.19}
I_{c^*}\leq I(u_0)\leq \liminf\limits_{n\rightarrow+\infty}I(u_n)=\liminf\limits_{n\rightarrow+\infty}I_{c_n}.
\end{equation}
On the other hand, for any $u\in \widetilde{S}(c^*)$, $\sqrt{\frac{c_n}{c^*}}u\in \widetilde{S}(c_n)$ and $\sqrt{\frac{c_n}{c^*}}u\rightarrow u$ in $\mathcal{H}.$
Then $\limsup\limits_{n\rightarrow+\infty}I_{c_n}\leq \lim\limits_{n\rightarrow+\infty}I(\sqrt{\frac{c_n}{c^*}}u)=I(u)$. Hence by the arbitrary of $u$, we have $\limsup\limits_{n\rightarrow+\infty}I_{c_n}\leq I_{c^*}$. Therefore, we conclude from \eqref{3.19} that $I(u_0)=I_{c^*},$ i.e. $u_0$ is a minimizer of $I_{c^*}$.\\

\noindent\textbf{Step~4}.~~For any $c>c^*$, $I_c=-\infty$ and there is no minimizer for \eqref{1.8}.

For any $c>c^*$ and $\rho>0$, let $x_0\in\R^3$ and $\psi$ be a radial cut-off function such that $\psi\equiv1$ on $B_1(0)$, $\psi\equiv0$ on $\R^3\backslash B_2(0)$, $0\leq\psi\leq1$ and $|\nabla\psi|\leq2$. Set
$$u^\rho(x)=A_\rho\sqrt{\frac{c}{c^*}}\rho^{\frac{3}{2}}\psi(x-x_0)Q(\rho(x-x_0)),$$
where $A_\rho>0$ is chosen to satisfy that $u^\rho\in \widetilde{S}(c)$. In fact, by the exponential decay \eqref{1.5} of $Q$, we have
\begin{equation}\label{3.25}
~~~~~~~~~\ds\frac{c}{A^2_{\rho}}=c+\ds\frac{c}{c^*}\ds\int_{\R^3}[\psi^2(\frac{x}{\rho})-1]Q^2(x)=c+O(\rho^{-\infty})~~~~\hbox{as}~\rho\rightarrow+\infty,
\end{equation} i.e. $A_\rho$ depends only on $\rho$ and $\lim\limits_{\rho\rightarrow+\infty}A_\rho=1.$ Here $O(\rho^{-\infty})$ denotes a function satisfying $\lim\limits_{\rho\rightarrow+\infty}O(\rho^{-\infty})\rho^s=0$ for all $s>0$. Since $V(x)\psi(x-x_0)$ has compact support, $D(u^\rho)\rightarrow \frac{V(x_0)c}{2}$. Then similarly to the proof of \eqref{3.25}, we have
$$I(u^\rho)=\ds\rho^2\frac{c A(Q)}{c^*}\left[1-\left(\frac{c}{c^*}\right)^{\frac{2}{3}}\right]+\frac{V(x_0)c}{2}
+\ds\rho\left(\frac{c}{c^*}\right)^2B(Q)+O(\rho^{-\infty})\rightarrow-\infty$$
as $\rho\rightarrow+\infty.$ So $e_c=-\infty$.

\end{proof}

\noindent $\textbf{Proof of Proposition \ref{prop1.3}}$\,\,\
\begin{proof}~~
(1)~~Similarly to Step 2 in the proof of Theorem \ref{th1.2}, we can show that the function $c\mapsto\frac{I_c}{c^2}$ is strict decreasing on $(0,c^*]$.

(2)~~Let us first show that
\begin{equation}\label{3.27}
\frac{I_c}{c^2}~\hbox{is~continuous~at~each}~c\in(0,c^*].
\end{equation}
By (1), to prove \eqref{3.27} is equivalent to prove that for any $c\in(0,c^*]$, $\{c_n\}\subset(0,c^*]$ such that $c_n\rightarrow c^-$, it must have
\begin{equation}\label{3.29}
\lim_{c_n\rightarrow c^-}\frac{I_{c_n}}{c_n^2}\leq\frac{I_{c}}{c^2}.\end{equation}

By Theorem \ref{1.2}, there exists $u_c\in \widetilde{S}(c)$ such that $I(u_c)=I_c$, let $v_n:=\sqrt{\frac{c_n}{c}}u_c$, then $v_n\in \widetilde{S}(c_n)$ and
 $$\frac{I_{c_n}}{c_n^2}\leq\frac{I(v_n)}{c_n^2}=\frac{ A(u_c)+D(u_c)}{cc_n}+\frac{ B(u_c)}{c^2}-\frac{C(u_c)}{c^{\frac{5}{3}}c_n^{\frac{1}{3}}}\rightarrow \frac{ I_c}{c^2},$$
 hence \eqref{3.29} holds and then \eqref{3.27} holds, i.e. if $c_n\rightarrow c$, then $\frac{I_{c_n}}{c_n^2}=\frac{I_c}{c^2}+o_n(1)$, where $o_n(1)\rightarrow0$ as $c_n\rightarrow c$, hence $\lim_{c_n\rightarrow c} I_{c_n}=\lim_{c_n\rightarrow c}\left(\frac{I_c}{c^2}c_n^2\right)=I_c.$
So $I_c$ is continuous on $(0,c^*]$.

 (3)~~Let $u_{c^*}\in \widetilde{S}(c^*)$ be a minimizer of $I_{c^*}$ and set $v:=\sqrt{\frac{c}{c^*}}u_{c^*}$ for any $c\in(0,c^*]$, then $v\in \widetilde{S}(c)$ and
$$
 0\leq I_c\leq I(v)=\ds\frac{c}{c^*}(A(u_{c^*})+D(u_{c^*}))+\ds\left(\frac{c}{c^*}\right)^2B(u_{c^*})
 -\left(\frac{c}{c^*}\right)^{\frac{5}{3}}C(u_{c^*})\rightarrow0
$$
 as $c\rightarrow0^+$. So $\lim\limits_{c\rightarrow0^+}I_c=0.$
 \end{proof}

\section{Proof of Theorems \ref{th1.4} and \ref{th1.5}}
Recall in Section 1 that up to translations, $Q$ is the unique positive least energy solution of the following equation
\begin{equation}\label{6.1}
-\Delta u+u=|u|^{\frac{4}{N}}u,~~~x\in\R^N.
\end{equation}
Define $J(u):H^1(\R^N)\rightarrow\R$ as follows:
$$J(u)=\frac{1}2\int_{\R^N}(|\nabla u|^2+u^2)-\frac{N}{2N+4}\int_{\R^N}|u|^{\frac{2N+4}{N}},$$
then by \eqref{1.13}, we see that
$$\inf\{J(u)|~u~\hbox{is~a~nontrivial~solution~of}~\eqref{6.1}\}=J(Q)=\frac{c^*}2.$$

\noindent $\textbf{Proof of Theorem \ref{th1.4}}$\,\,\
\begin{proof}~~The proof of $(1)(2)$ is similar to that of Theorem \ref{1.1}. Let us next prove $(3)-(5)$.

If there exists some $c>0$ such that $F|_{S(c)}$ has a critical point $u_c\in S(c)$, then there is a Lagrange multiple $\lambda_c\in \R$ such that $(u_c,\lambda_c)$ satisfies the following equation
\begin{equation}\label{5.1}
-\Delta u-|u|^{\frac{4}{N}}u=\lambda_cu,~~~~x\in\R^N.
\end{equation}
Moreover, $u_c$ satisfies the Pohozaev identity:
\begin{equation}\label{5.2}
\frac{N-2}{2}\ds\int_{\R^N}|\nabla u_c|^2-\frac{N^2}{2N+4}\ds\int_{\R^N}|u_c|^{\frac{2N+4}{N}}=\frac{N}{2}\lambda_cc.
\end{equation}
So,
\begin{equation}\label{5.3}
\ds\int_{\R^N}|u_c|^{\frac{2N+4}{N}}=\frac{N+2}{N}\ds\int_{\R^N}|\nabla u_c|^2=-\frac{N+2}{2}\lambda_cc,
\end{equation}
which implies that $\lambda_c<0$. Moreover, by \eqref{1.6}, we have $c\geq c^*$.
Therefore, if $c<c^*$, then $F(u)$ has no critical point restricted on $S(c)$.

 If $c=c^*$, then by \eqref{1.13} we see that $f_{c^*}$ is attained by $Q$. Furthermore, if $u_{c^*}$ is a critical point of $F|_{S(c^*)}$, then similarly, there exists $\lambda_{c^*}<0$ such that \eqref{5.1}-\eqref{5.3} hold. Set $u_{c^*}(x)=(\sqrt{-\lambda_{c^*}})^{\frac{N}{2}}w_{c^*}(\sqrt{-\lambda_{c^*}}x)$, then $w_{c^*}$ is a nontrivial solution of \eqref{6.1} and
$J(w_{c^*})=\frac{|w_{c^*}|_2^2}{2}=\frac{c^*}{2},$ i.e. $w_{c^*}$ is a least energy solution of \eqref{6.1}. Then up to translations, $w_{c^*}=Q$.

To show (4), by contradiction, if for some $c>c^*$, the critical point $u_c$ has constant sign. We may assume that $u_c\geq0$. By the strong maximum principle, $u_c>0.$ Then similarly, there exists $\lambda_{c}<0$ such that $w_c(x)=(\sqrt{-\lambda_{c}})^{-\frac{N}{2}}u_{c}(\frac{x}{\sqrt{-\lambda_{c}}})$ is a positive solution of \eqref{6.1}, then $w_c=Q$, so $c=|u_c|_2^2=|w_c|_2^2=|Q|_2^2=c^*$, which is impossible. Then the theorem is proved.
\end{proof}

We next consider the minimization problem \eqref{1.14}:
$$f_\mu(c)=\inf\limits_{u\in S(c)}F_\mu(u),$$
where
$$F_\mu(u)=\frac12\ds\int_{\R^N}|\nabla u|^2-\frac\mu2\ds\int_{\R^N}V(x)u^2-\frac{N}{2N+4}\ds\int_{\R^N}|u|^{\frac{2N+4}{N}}$$
and $\mu>0$ and $V(x)$ satisfies $(V_2)$.
It is easy to see from $(V_2)$ that $V(x)$ is bounded a.e. in $\R^N$, i.e. there exists some $V_0>0$ such that $0\leq V(x)\leq V_0$ a.e. in $\R^N.$

We recalled in Section 1 that for a bounded domain $\Omega\in \R^N$ with smooth boundary and $V(x)\not\equiv0$ a.e. in $\Omega,$ the minimum problem
\begin{equation}\label{4.1}
\mu_1:=\inf\{\ds\int_{\R^N}|\nabla u|^2|~u\in H^1_0(\Omega),\ds\int_{\Omega}V(x)u^2=1\}
\end{equation}
is achieved by $\phi\in H_0^1(\Omega)$ with $\int_{\Omega}V(x)\phi^2=1$ and $\phi>0$ a.e. in $\Omega$.

\begin{lemma}\label{lem4.1}~~Suppose that $\mu>0$ and $V(x)$ satisfies $(V_2)$,

$(1)$~~$F_\mu(u)$ is bounded from below on $S(c)$ for all $c\in (0,c^*]$; $f_\mu(c)=-\infty$ for all $c>c^*$. Moreover, $F_\mu(u)$ is coercive on $S(c)$ for $0<c<c^*$.

$(2)$~~for any $c\in(0,c^*]$, $f_\mu(c)\leq 0$, moreover, $f_\mu(c)<0$ if $\mu\geq\mu_1$.
\end{lemma}
\begin{proof}~~$(1)$~~For any $c\in(0,c^*]$ and any $u\in S(c)$, by $(V_2)$, we have that
\begin{equation}\label{4.2}
F_\mu(u)\geq\frac{1}{2}\left(1-\left(\frac{c}{c^*}\right)^{\frac{2}{N}}\right)\ds\int_{\R^N}|\nabla u|^2
-\frac{\mu}{2}\ds\int_{\R^N}V(x)u^2\geq -\frac{V_0\mu c}{2},
\end{equation}
then $f_\mu(c)>-\infty$ for all $c\in(0,c^*]$. Moreover, we see from \eqref{4.2} that $F_\mu(u)$ is coercive on $S(c)$ if $0<c<c^*$.

For $c>c^*$, let $Q^t(x):=\sqrt{\frac{c}{c^*}}t^{\frac{N}{2}}Q(tx)$ with $t>0$, where $Q$ is given in \eqref{1.11}, then by $(V_2)$, we have
$$F_\mu(Q^t)=\frac{t^2}{2}\frac{c\int_{\R^N}|\nabla Q|^2}{c^*}\left[1-\left(\frac{c}{c^*}\right)^{\frac{2}{N}}\right]-\frac{\mu c}{2c^*}\ds\int_{\R^N}V(\frac{x}{t})Q^2\rightarrow-\infty~~~\hbox{as}~~t\rightarrow +\infty,$$
which implies that $f_\mu(c)=-\infty$ for each $c>c^*$.

$(2)$~~For $c\in (0,c^*]$ and $u\in S(c)$, set $u^t(x)=t^{\frac{N}{2}}u(tx)$, then $u^t\in S(c)$ and by $(V_2)$, we have
 $$F_\mu(u^t)=\frac{t^2}{2}\ds\int_{\R^N}|\nabla u|^2-\frac{\mu}{2}\ds\int_{\R^N}V(\frac{x}{t})u^2-\frac{Nt^2}{2N+4}\ds\int_{\R^N}|u|^{\frac{2N+4}{N}}\rightarrow 0~~\hbox{as}~t\rightarrow0^+,$$
 hence $f_\mu(c)\leq0$.

If $\mu\geq\mu_1$, set $\phi_c:=\frac{\sqrt{c}\phi}{|\phi|_2}$, where $\phi$ is given in \eqref{4.1}, then $\phi_c\in S(c)$ and by \eqref{4.1}, we see that
$$\begin{array}{ll}
F_\mu(\phi_c)&=\ds\frac{c}{2|\phi|_2^2}\left(\ds\int_{\R^N}|\nabla \phi|^2-\mu\ds\int_{\R^N}V(x)\phi^2\right)-\frac{N}{2N+4}\left(\frac{\sqrt{c}}{|\phi|_2}\right)^{\frac{2N+4}{N}}\ds\int_{\R^N}|\phi|^{\frac{2N+4}{N}}\\[5mm]
&=\ds\frac{c}{2|\phi|_2^2}\left(\ds\int_{\Omega}|\nabla \phi|^2-\mu\ds\int_{\Omega}V(x)\phi^2\right)-\frac{N}{2N+4}\left(\frac{\sqrt{c}}{|\phi|_2}\right)^{\frac{2N+4}{N}}\ds\int_{\Omega}|\phi|^{\frac{2N+4}{N}}\\[5mm]
&=\ds\frac{c}{2|\phi|_2^2}\left(\mu_1-\mu\right)-\frac{N}{2N+4}\left(\frac{\sqrt{c}}{|\phi|_2}\right)^{\frac{2N+4}{N}}\ds\int_{\Omega}|\phi|^{\frac{2N+4}{N}}\\[5mm]
&<0,
\end{array}$$
which shows that $f_\mu(c)<0$ for all $c\in(0,c^*]$ if $\mu\geq\mu_1$.
\end{proof}

\begin{lemma}\label{lem4.2}~~
Suppose that $\mu>0$ and $V(x)$ satisfies $(V_2)$, then the function $c\mapsto f_\mu(c)$ is continuous on $(0,c^*)$.
\end{lemma}
\begin{proof}
The proof is similar to that of Theorem 2.1 in \cite{bs2}. For readers' convenience, we give its detailed proof.

It is enough to show that if $c\in(0,c^*)$ and $\{c_n\}\subset(0,c^*)$ such that $c_n\rightarrow c$ as $n\rightarrow+\infty$, then
\begin{equation}\label{4.3}
\lim\limits_{n\rightarrow+\infty}f_\mu(c_n)=f_\mu(c).
\end{equation}
Let $\{u_n\}\subset S(c_n)$ and $\{v_n\}\subset S(c)$ such that
$$F_\mu(u_n)<f_\mu(c_n)+\frac1n$$
and $$F_\mu(v_n)\rightarrow f_\mu(c)~~\hbox{as}~n\rightarrow+\infty,$$
then by Lemma \ref{lem4.1},
$\{u_n\}$ and $\{v_n\}$ are uniformly bounded in $H^1(\R^N)$ respectively. Hence
$$\begin{array}{ll}
f_\mu(c)&\leq F_\mu(\sqrt{\frac{c}{c_n}}u_n)\\[5mm]
&=\ds\frac{c}{c_n}\left(\ds\frac12\ds\int_{\R^N}|\nabla u_n|^2-\frac{\mu}{2}\ds\int_{\R^N}V(x)|u_n|^2\right)
-\frac{N}{2N+4}\left(\frac{c}{c_n}\right)^{\frac{N+2}{N}}\ds\int_{\R^N}|u_n|^{\frac{2N+4}{N}}\\[5mm]
&=F_\mu(u_n)+o_n(1)\leq f_\mu(c_n)+o_n(1),
\end{array}$$
where $o_n(1)\rightarrow 0$ as $n\rightarrow+\infty$. On the other hand,
$$f_\mu(c_n)\leq F_\mu(\sqrt{\frac{c_n}{c}}v_n)=F_\mu(v_n)+o_n(1)\rightarrow f_\mu(c).$$
So \eqref{4.3} is proved.
\end{proof}

\begin{lemma}\label{lem4.3}~~Suppose that $V(x)$ satisfies $(V_2)$ and $\mu\geq\mu_1$, then for all $0<c<c^*$, we have
$$f_\mu(c)<f_\mu(\alpha)+f_\mu(c-\alpha),~~~\forall~0<\alpha<c.$$
\end{lemma}
\begin{proof}
For $0<c<c^*$, by Lemma \ref{lem4.1} $(2)$, we see that $f_\mu(c)<0$. Let $\{u_n\}\subset S(c)$ be a minimizing sequence of $f_\mu(c)$, then Lemma \ref{lem4.1} (1) shows that $\{u_n\}$ is bounded in $H^1(\R^N)$. Moreover, there exists $k_1>0$ independent of $n$ such that
\begin{equation}\label{4.4}
\ds\int_{\R^N}|u_n|^{\frac{2N+4}{N}}\geq k_1.
\end{equation}
Indeed, if $\int_{\R^N}|u_n|^{\frac{2N+4}{N}}\rightarrow 0$, then for any $\varepsilon>0$, there exists $n_0=n_0(\varepsilon)>0$ such that $\int_{\R^N}|u_n|^{\frac{2N+4}{N}}<\varepsilon$ for all $n>n_0$. By $(V_2)$, there exists $R=R(\varepsilon)>0$ such that $0\leq V(x)<\varepsilon$ for all $|x|\geq R$. Then for $n>n_0$, there exist constants $C_1,C_2>0$ independent of $n$ such that
$$
\begin{array}{ll}
\ds\int_{\R^N}V(x)|u_n|^2&=\ds\int_{B_R(0)}V(x)|u_n|^2+\ds\int_{\R^N\backslash B_R(0)}V(x)|u_n|^2\\[5mm]
&\leq V_0C_1|u_n|_{\frac{2N+4}{N}}^2+\varepsilon\ds\int_{\R^N\backslash B_R(0)}|u_n|^2\\[5mm]
&\leq(V_0C_1+C_2)\varepsilon,
\end{array}$$
hence $\int_{\R^N}V(x)|u_n|^2\rightarrow0$ by the arbitrary of $\varepsilon$ and $V(x)\geq0$. So $$f_\mu(c)=\lim\limits_{n\rightarrow+\infty}F_\mu(u_n)=\lim\limits_{n\rightarrow+\infty}\frac12\ds\int_{\R^N}|\nabla u_n|^2\geq0,$$ which is a contradiction.

Set $u_n^\theta:=\sqrt{\theta} u_n$ with $\theta>1$, then $u_n^\theta\in S_{\theta c}$ and by \eqref{4.4}, we have
$$F_\mu(u_n^\theta)-\theta F_\mu(u_n)=\frac{N}{2N+4}(\theta-\theta^{\frac{N+2}{N}})\ds\int_{\R^N}|u_n|^{\frac{2N+4}{N}}
\leq-\frac{Nk_1(\theta^{\frac{N+2}{N}}-\theta)}{2N+4}<0,$$
which implies that $f_\mu(\theta c)<\theta f_\mu(c)$ by letting $n\rightarrow+\infty.$ Then we easily conclude our result and the lemma is proved.
\end{proof}

\noindent $\textbf{Proof of Theorem \ref{th1.5}}$\,\,\
\begin{proof}~~$(1)$~~The proof of $(1)$ is given in Lemma \ref{lem4.1} $(1)$.

$(2)$~~For any $0<c<c^*$, by Lemma \ref{lem4.1} (2), $f_\mu(c)<0$. Let $\{u_n\}\subset S(c)$ be a minimizing sequence of $f_\mu(c)$, then by Lemma \ref{lem4.1} (1), $\{u_n\}$ is bounded in $H^1(\R^N)$. Hence we may assume that there exists $u_c\in H^1(\R^N)$ such that
\begin{equation}\label{4.5}
\left\{
 \begin{array}{ll}
 \ds u_n\rightharpoonup u_c,\,\,\, &\hbox{in}~H^1(\R^N), \vspace{0.2cm}\\
 \ds u_n\rightarrow u_c,\,\,\, &\hbox{in}~L^q_{loc}(\R^N),\ \ q\in[1,2^*),\\
 \ds u_n(x)\rightarrow u_c(x),\,\,\, &\hbox{a.e.~in}~\R^N.
 \end{array}
 \right.
 \end{equation}
Moreover, $u_c\not\equiv0$. By contradiction we just suppose that $u_c\equiv0$. By $(V_2)$, for any $\varepsilon>0$, there exists $R>0$ such that $0\leq V(x)<\varepsilon$ for all $|x|\geq R$. Then there exists a constant $C>0$ such that $\int_{\R^N\backslash B_R(0)}V(x)|u_n|^2<C\varepsilon$. We see from \eqref{4.5} and $(V_2)$ that $\int_{B_{R}(0)}V(x)|u_n|^2\rightarrow0$ as $n\rightarrow+\infty$. Hence $\int_{\R^N}V(x)|u_n|^2\rightarrow 0$. So by \eqref{1.6}, we have
$$f_\mu(c)=\lim\limits_{n\rightarrow+\infty}F_\mu(u_n)\geq\lim\limits_{n\rightarrow+\infty}-\frac{\mu}2\ds\int_{\R^N}V(x)|u_n|^2=0,$$
which is a contradiction. Therefore $\alpha:=|u_c|_2^2\in(0,c]$.

We next show that $u_c\in S(c)$. By contradiction, if $0<\alpha<c$, then $u_c\in S_\alpha$. By \eqref{4.5}, we have
\begin{equation}\label{4.6}
|u_n|_2^2=|u_c|_2^2+|u_n-u_c|_2^2+o_n(1).
\end{equation}
By the Brezis-Lieb Lemma and Lemma \ref{lem4.2}, we see that
$$f_{\mu}(c)=\lim\limits_{n\rightarrow+\infty}F_{\mu}(u_n)
= F_{\mu}(u_c)+\lim\limits_{n\rightarrow+\infty}F_{\mu}(u_n-u_c)\geq f_{\mu}(\alpha)+f_{\mu}(c-\alpha),$$
which contradicts Lemma \ref{lem4.3}. Then $|u_c|_{2}^2=c$. So $u_c\in S_c$ and then by the interpolation inequality, $u_n\rightarrow u_c$ in $L^p(\R^N),$ $p\in[2,2^*)$. Hence by \eqref{4.5},
$$f_{\mu}(c)\leq F_{\mu}(u_c)\leq \lim\limits_{n\rightarrow+\infty}F_{\mu}(u_n)=f_{\mu}(c).$$
So $u_c$ is a minimizer of $f_{\mu}(c)$ and then $u_c$ is a constraint critical point of $F_{\mu}$ on $S(c)$. Therefore, there exists $\lambda_c\in\R$ such that $ F'_{\mu}(u_c)-\lambda_cu_c=0$ in $H^{-1}(\R^N)$, i.e. $(u_c,\lambda_c)$ is a couple of solution to the following equation
$$-\Delta u-\mu V(x)u-|u|^{\frac{4}{N}}u=\lambda_cu~~~\hbox{in}~\R^N.$$
Moreover, the fact that $F_\mu(u_c)<0$ shows that
$$
\lambda_cc=\langle F'_{\mu}(u_c),u_c\rangle=2F_{\mu}(u_c)-\frac{2}{N+2}\ds\int_{\R^N}|u_c|^{\frac{2N+4}{N}}<0,
$$
i.e. $\lambda_c<0$.

\end{proof}




\begin{thebibliography}{99}
\addcontentsline{toc}{section}{\protect \heiti 参考文献} 
\bibitem{aru}
A. Ambrosetti, D. Ruiz, Multiple bound states for the
Schr\"{o}dinger-Poisson problem, Commun. Contemp. Math. 10 (3)
(2008), 391-404.

\bibitem{dm}
T. D'Aprile, D. Mugnai, Solitary waves for nonlinear Klein-Gordon-Maxwell and Schr\"{o}dinger-Maxwell equations, Proc. R. Soc. Edinb. Sect. A 134 (5) (2004), 893-906.


\bibitem{apa}
A. Azzollini, A. Pomponio, P. d'Avenia, On the Schr\"{o}dinger-Maxwell equations under the effect of a general nonlinear term, Ann. Inst. H. Poincar\'{e} Anal. Non Lin\'{e}aire 27 (2) (2010), 779-791.


\bibitem{bgg}
F. Bardos, A. Golse, D. Gottlieb, N. Mauser, Mean field dynamics of fermions and the time-dependent Hartree-Fock equation, J. Math. Pures Appl. 82 (6) (2003), 665-683.

\bibitem{bw}
T. Bartsh, Z. Q. Wang, Existence and multiplicity results for some superlinear elliptic problems on $\R^N$, Comm. Partial Differ. Equ. 20 (1995), 1725-1741.

\bibitem{bjl}
J. Bellazzini, L. Jeanjean, T. J. Luo, Existence and instability of standing waves with prescribed norm for a class of Schr\"{o}dinger-Poisson equations, Proc. London Math. Soc. 107 (3) (2013), 303-339.



\bibitem{bs2}
J. Bellazzini, G. Siciliano, Scaling properties of functionals and existence of constrained minimizers, J. Funct. Anal. 261 (9) (2011), 2486-2507.


\bibitem{bs1}
J. Bellazzini, G. Siciliano, Stable standing waves for a class of nonlinear Schr\"{o}dinger-Poisson equations, Z. Angew. Math. Phys. 62 (2) (2011), 267-280.

\bibitem{cdss}
I. Catto, J. Dolbeault, O. S\'{a}nchez, J. Soler, Existence of steady states for the Maxwell-Sch\"{o}dinger-Poisson system: exploring the applicability of the concentration-compactness principle, Math. Models Methods Appl. Sci. 23 (2013), 1915-1938.

\bibitem{gnn}
B. Gidas, W. M. Ni, L. Nirenberg, Symmetry of positive solutions of nonlinear elliptic equations in $\R^n$, Mathematical analysis and application Part A, Adv. in Math. Suppl. Stud. vol. 7, Academic Press, New York, (1981), 369-402.

\bibitem{gs}
Y. J. Guo, R. Seiringer, On the mass concentration for Bose-Einstein condensates with attractive interactions, Lett. Math. Phys. 104 (2014), 141-156.

\bibitem{jean}
L. Jeanjean, Existence of solutions with prescribed norm for semilinear elliptic equations, Nonlinear Anal. 28 (10) (1997), 1633-1659.

\bibitem{jl}
L. Jeanjean, T. J. Luo, Sharp nonexistence results of prescribed $L^2$-norm solutions for some class of Schr\"{o}dinger-Poisson and quasi-linear equations, Z. Angrew. Math. Phys. 64 (2013), 937-954.

\bibitem{k1}
H. Kikuchi, Existence and stability of standing waves for Schr\"{o}dinger-Poisson-Slater equation, Adv. Nonlinear Stud. 7 (3) (2007), 403-437.

\bibitem{k2}
H. Kikuchi, Existence and orbital stability of the standing waves for nonlinear Schr\"{o}dinger equations via the variational method, Doctoral Thesis (2008).

\bibitem{k}
M. K. Kwong, Uniqueness of positive solutions of $\Delta u-u+u^p=0$ in $\R^N$, Arch. Rational Mech. Anal. 105 (1989), 243-266.

\bibitem{ls}
E. H. Lieb, B. Simon, The Thomas-Fermi theory of atoms, molecules, and solids, Adv. Math. 23 (1) (1977), 22-116.

\bibitem{li}
P. L. Lions, Solutions of Hartree-Fock equations for Coulomb systems, Commun. Math. Phys. 109 (1) (1987), 33-97.

\bibitem{m}
N. J. Mauser, The Schr\"{o}dinger-Poisson-X$\alpha$ equation, Appl. Math. Lett. 14 (6) (2001), 759-763.





\bibitem{r1}
D. Ruiz, The Schr\"{o}dinger-Poisson equation under the effect of a
nonlinear local term, J. Funct. Anal. 237 (2) (2006), 655-674.

\bibitem{r2}
D. Ruiz, On the Schr\"{o}dinger-Poisson-Slater system: behavior of minimizers, radial and nonradial cases, Arch. Rational Mech. Anal. 198 (1) (2010), 349-368.




\bibitem{ss}
O. Sanchez, J. Soler, Long time dynamics of the Schr\"{o}dinger-Poisson-Slater system, J. Stat. Phys. 114 (1-2) (2004), 179-204.





\bibitem{stuart}
C. A. Stuart, Bifurcation from the essential spectrum norm for semilinear elliptic linearities, Math. Methods Appl. Sci. 11 (1989), 525-542.



\bibitem{sw}
J. T. Sun, T. F. Wu, Ground state solutions for an indefinite Kirchhoff type problem with steep potential well, J.
Differ. Equ. 256 (2014), 1771-1792.

\bibitem{w}
M. I. Weinstein, Nonlinear Schr\"{o}dinger equations and sharp interpolations estimates, Commun. Math. Phys. 87 (1983), 567-576.


\bibitem{wi}
M. Willem,  Minimax theorems, Progress in Nonlinear Differential
Equations and their Applications, 24. Birkh\"{a}user Boston, Inc.,
Boston, MA, 1996.

\end{thebibliography}
 \end{document}